\documentclass[12pt,a4paper]{article}
\usepackage{amsfonts}
\usepackage{epsfig}
\usepackage{graphicx}
\usepackage{amsmath}
\usepackage{amssymb}
\usepackage{a4wide}

\newcounter{main}

\newtheorem{theorem}{Theorem}[section]
\newtheorem{proposition}[theorem]{Proposition}
\newtheorem{lemma}[theorem]{Lemma}
\newtheorem{corollary}[theorem]{Corollary}

\newtheorem{definition}{Definition}[section]

\newcommand{\blanksquare}{\,\,\,$\sqcup\!\!\!\!\sqcap$}
\newenvironment{proof}{{\flushleft {\bf Proof: }}}{\blanksquare}

\newcounter{example}
{{\stepcounter{example}}{\flushleft {\bf Example \arabic{example}:}}}%
{\par}
\def\ds{\displaystyle}

\title{{\bf An ergodic theorem for non-invariant measures }}
\author{ Maria Carvalho and Fernando Moreira \thanks{Partially supported by FCT through
CMUP }}
\date{ }

\begin{document}

\maketitle

\begin{abstract}
Given a space $X$, a $\sigma$-algebra $\mathfrak{B}$ on $X$ and a measurable map $T:X \rightarrow X$, we say that a measure $\mu$ is \emph{half-invariant} if, for any $B \in \mathfrak{B}$, we have $\mu(T^{-1}(B)\leq \mu (B)$. In this note we present a generalization of Birkhoff's Ergodic theorem to $\sigma$-finite half-invariant measures.
\end{abstract}

\emph{keywords:} Birkhoff's averages; pointwise convergence \\
\normalsize

\begin{section}{Introduction}
In the study of properties of systems that hold on big sets of the domain, it is useful to connect dynamics with measures. The classical and most powerful relation between them is the notion of \emph{invariance}, meaning that each measurable set has the same measure as its pre-image by the dynamics. The importance of invariant probabilities may be attested by Poincar\'e's recurrence theorem which ensures that the existence of a finite invariant measure replaces, for certain purposes, the compactness of the domain. In particular, if $\mu$ is a $T$-invariant probability and $A$ a measurable set, then $\mu$-almost every point $a$ in $A$ is recurrent. This result is not valid for infinite measures and gives no clue about the frequency of visits to $A$ or how that rate changes with the starting point $a$ or the set $A$. This is done by Birkhoff's ergodic theorem.

\begin{theorem}\cite{Ha}
Consider a $\sigma$-finite measure $\mu$, a measurable space $(X,\mathfrak{B},\mu)$, a measu\-rable map $T:X \rightarrow X$ and a function $f:X \rightarrow \mathbb{R}$ in $\mathcal{L}^1(X,\mu)$. Then:
\begin{itemize}
\item[(a)] The limit $f^\ast(x) = \underset{n\rightarrow{\infty}}{\text{lim}} \, \frac{1}{n} \ds\sum_{j=0}^{n-1} f(T^j(x))$ exists for $\mu$ almost every point $x$.
\item[(b)] The function $f^\ast$ is $\mu$-integrable and $T$-invariant.
\item[(c)] If $\mu(X)<+\infty$, then $\ds\int \, f^\ast d\mu=\ds\int \, f d\mu$.
\item[(d)] If $0<\mu(X)<+\infty$ and $\mu$ is ergodic, then $f^\ast(x)=\frac{1}{\mu(X)}\ds\int \, f d\mu$ for $\mu$ almost every $x$.
\end{itemize}
\end{theorem}

Yet, there are dynamics without invariant probabilities. The possibility of generalizing this result lies on what we request about one or several of the ingredients it concerns, namely the dynamics, the measure or the test function. For instance, extensions have been achieved demanding more from the dynamics $T$ (like Halmos's Ergodic theo\-rem in \cite{H} and Halmos's Random Ergodic theorem in \cite{Ha}) or weighing the values of the test function along the orbits (as done by Khintchine in \cite{Kh} and by Wiener \& Wintner in \cite{WW}) or both (as Hurewicz did in \cite{Hu}). The condition of measure invariance plays an essential role on the proof of Birkhoff's theorem and it is a delicate matter to extend this result to a broader set of measures.

We find in the literature several examples of dynamical systems and test functions with time averages failing to converge almost everywhere (\cite{LS,T,G}). For instance, as in \cite{LS}, we may consider a uniquely ergodic homeomorphism $T:X \rightarrow X$ on a compact metric space, a point $p \in X$ and a real valued continuous map $g$ whose sequence of time averages along the orbit of $p$ with respect to $T^2$ does not converge; therefore, if $\eta=\frac{1}{2}(\delta_{p}+\delta_{T(p)})$, then $\eta(\{p\})>0$, the sequence
$$\frac{1}{n}\ds\sum_{j=0}^{n-1}\ds\int g((T^{2})^{j})\, d\eta = \frac{1}{2}\left[\frac{1}{n}\ds\sum_{j=0}^{2n-1} g(T^{j})(p)\right]$$
is convergent (by the unique ergodicity of $T$), but $\left(\frac{1}{n}\ds\sum_{j=0}^{n-1} g((T^{2})^{j})(p)\right)_{n \in \mathbb{N}}$ diverges. Besides, even if the sequence of time averages converges almost everywhere, the set of points where convergence fails may be topologically significant: if $T: X \rightarrow X$ is the Anosov diffeomorphism on the 2-torus determined by the matrix
$\left( \begin{array}{cc}
2 & 1 \\
1 & 1
\end{array}
\right)$
and $f$ is a real valued non-constant continuous map, then the sequence $\left(\frac{1}{n}\ds\sum_{j=0}^{n-1} f(T^j(x))\right)_{n \in \mathbb{N}}$ converges in a set $B$ with Lebesgue measure one, but its complement is residual \cite{M}. So, if we take any point $a$ in $X-B$ and the Dirac measure supported on the orbit $a$, then the limit of the Birkhoff averages of $f$ does not exist almost everywhere.

Thus, to generalize Birkhoff's ergodic theorem to non-invariant measures, we need some extra condition on the measure. The first promising hint was the \emph{quasi-invariance}. In a measurable space $(X,\mathfrak{B})$, a measure $\mu$ is \emph{quasi-invariant} by a measurable transformation $T:X \rightarrow X$ if
\begin{equation*}
\forall B \in \mathfrak{B}\,\,\,\,\,\,\,\mu(B)=0 \Leftrightarrow \mu(T^{-1}(B))=0.
\end{equation*}
For example, given $T:X \rightarrow X$ and $a \in X$, the probability $\eta_a$, supported on the orbit of $a$, which assigns to each measurable set $A$ the sum $\ds\sum_{n \in \mathbb{N}: \,\,T^n(a) \in A} \, \, \frac{1}{2^{n+1}}$, is quasi-invariant and ergodic (but in general is not $T$-invariant, as happens with $T(x)=\frac{x}{2}$, $0\leq x \leq 1$, and $a=1$). Quasi-invariance is a relevant property: an ergodic decomposition is available for these probabilities and, moreover, it had been proven that:
\begin{itemize}
 \item[] \emph{Every uniquely ergodic homeomorphism of a compact metric space whose invariant measure is non-atomic possesses uncountably many inequivalent non-atomic quasi-invariant ergodic measures.} \cite{Kr}
\item[] \emph{A necessary and sufficient condition for a homeomorphism of a compact metric space to have uncountably many inequivalent non-atomic quasi-invariant measures is that the dynamics has a point that returns infinitely often to any of its deleted neighborhoods.} \cite{KW}
\end{itemize}
Nevertheless there is no hope to extend Birkhoff's ergodic theorem (or Poincar\'e's result) to this setting. Let us go back to the previous example,
$$T(x)=\frac{x}{2}, \,\, 0\leq x \leq 1; \,\,\,\, a=1; \,\,\,\, \eta_a$$
and take a sequence $(\lambda_k)_{k \in \mathbb{N}_{0}}$ of zeros and ones whose averages $\left(\frac{1}{n}\ds\sum_{j=0}^{n-1} \, \lambda_j\right)_{n \in \mathbb{N}}$ do not converge. Then the test map $f$ defined as
$$f(x)=\left\{ \begin{array}{cc}
0 & \text{ if } x \text{ does not belong to the orbit of } 1 \\
\lambda_n & \text{ if } x=T^n(1)
\end{array}
\right.
.$$
verifies:
\begin{itemize}
\item $f \in \mathcal{L}^1(X,\mu)$, since $\ds\int \, f d\mu=\ds\sum_{n \geq 0} \, \frac{f(T^n(1))}{2^{n+1}} \leq 1$;
\item $\left(\frac{1}{n}\ds\sum_{j=0}^{n-1} \, f(T^j(1))\right)_{n \in \mathbb{N}}=\left(\frac{1}{n}\ds\sum_{j=0}^{n-1} \, \lambda_j\right)_{n \in \mathbb{N}}$ has no limit;
\item $\mu(\{1\})=\frac{1}{2}$;
\item the set $\{1\}$ has no recurrent point.
\end{itemize}
So we turned to a slightly stronger concept.

\begin{definition}
A measure $\mu$ on a $\sigma$-algebra $\mathfrak{B}$ is \emph{half-invariant} by $T$ if
$$\forall \,B \in \mathfrak{B} \,\,\, \, \mu(T^{-1}(B)) \leq \mu(B).$$
\end{definition}

\noindent Invariant measures are half and quasi-invariant, and half-invariance implies quasi-invariance, but the converse is not true. For instance, if $X=\mathbb{R}$ and $T: \mathbb{R} \rightarrow \mathbb{R}$, $T(x)=2x$, then $\mu= \text{\emph{Lebesgue measure}}$ is half-invariant but not invariant: for any measurable $B$, we have $\mu\left(T^{-1}(B)\right)=\frac{\mu(B)}{2}$.

\medskip

There are a few reasons to choose this concept:
\begin{enumerate}
\item We may find examples.

$(i)$ For linear maps in $\mathbb{R}^n$ with determinant bigger than $1$, the \emph{Lebesgue measure} is half-invariant.

\medskip

$(2i)$ If $T:\mathbb{Z} \rightarrow \mathbb{Z}$ is the map $T(n)=n-1$ and $\mathfrak{B}=\{\text{subsets of }\mathbb{Z}\}$, then the\emph{ weighted counting measure} $\nu$ given by $\nu(A)=\ds\sum_{n \in A}\, \frac{1}{2^n}$ is half-invariant. [Notice that its finite counterpart, the probability $\eta$ defined as
$\eta(A)=\ds\sum_{n \in A}\, \frac{1}{3\times 2^{|n|}}$ is quasi-invariant but is not half-invariant.]

\item They seem to be naturally associated with \emph{sub-Markov} or Frobenius operators, \emph{random diffusion} equations and other relevant subjects, as suggested by \cite{Ru}.
\item It is a spectral property.

A measure $\mu$ is half-invariant by $T$ if and only if the operator $\mathcal{U}_T:\mathcal{L}^1(X,\mu) \rightarrow \mathcal{L}^1(X,\mu)$, that assigns to each $f$ the composition $f\circ T$, is a contraction on $\mathcal{L}^1(X,\mu)$. That is, $\mu$ is half-invariant by $T$ if and only if, for any non-negative $f$, we have $\ds\int \, \mathcal{U}_T(f) \, d\mu \leq \ds\int \, f \, d\mu$.
\end{enumerate}

\bigskip

Now, for a probability $\mu$ and a positive contraction $\mathcal{U}$ on $\mathcal{L}^1(X,\mu)$, it was already known \cite{LS,Hg} that the two following assertions are equivalent:
\begin{enumerate}
\item[I.] \emph{For each $f$ in $\mathcal{L}^{\infty}$, the sequence $\left(\frac{1}{n}\ds\sum_{j=0}^{n-1} \mathcal{U}^j(f)\right)_{n \in \mathbb{N}}$ converges $\mu$ almost everywhere.}
\item[II.] \emph{For each $f$ in $\mathcal{L}^{\infty}$, the sequence $\left(\frac{1}{n}\ds\sum_{j=0}^{n-1} \ds\int \, \mathcal{U}^j(f)\, d\mu(x) \right)_{n \in \mathbb{N}}$ converges.}
\end{enumerate}
Obviously, if $\mathcal{U}(f)=f \circ T$ and $\mu$ is $T$-invariant, (I) and (II) are not only equivalent but both valid, and the limit in (II) is $\ds\int \, f \, d\mu$.
Moreover, (II) is a consequence of (I) for each fixed $f$, by the Dominated Convergence theorem. The converse is not so straightforward and needs the full extent of the hypothesis, that is, that (II) holds for any $f$. Chacon described in \cite{C} an example of a positive contraction $\mathcal{U}$ on some $\mathcal{L}^1(X,\mu)$ and a positive bounded and integrable function $f$ such that
$$\left(\frac{1}{n}\ds\sum_{j=0}^{n-1} \ds\int \, \mathcal{U}^j(f)\, d\mu(x) \right)_{n \in \mathbb{N}}=\left(\ds\int \, f \, d\mu(x) \right)_{n \in \mathbb{N}}$$
thus condition (II) holds, but the sequence $\left(\frac{1}{n}\ds\sum_{j=0}^{n-1} \mathcal{U}^j(f)(x)\right)_{n \in \mathbb{N}}$ fails to converge on a full Lebesgue measure set since, for almost every $x$,
\begin{eqnarray*}
\liminf_{n \rightarrow +\infty} \, \frac{1}{n}\ds\sum_{j=0}^{n-1} \mathcal{U}^j(f)(x)&=&0 \\
\limsup_{n \rightarrow +\infty} \, \frac{1}{n}\ds\sum_{j=0}^{n-1} \mathcal{U}^j(f)(x)&=&+ \infty.
\end{eqnarray*}
Therefore we restricted the study to operators associated to a dynamical system, that is, those defined as $\mathcal{U}(f)=f \circ T$, for some map $T:X \rightarrow X$.

\medskip

\begin{theorem}\label{maintheorem}
Let $(X,\mathfrak{B},\mu)$ be a measure space, $T:X \rightarrow X$ a measurable transformation and assume that $\mu$ is
a $\sigma$-finite measure half-invariant by $T$. Then, for any non-negative $f \in \mathcal{L}^1(X,\mu)$, we have:
\begin{itemize}
\item[(a)] The limit $f^\ast(x) = \underset{n\rightarrow{\infty}}{\text{lim}} \, \frac{1}{n} \ds\sum_{j=0}^{n-1} f (T^j(x))$ exists for $\mu$ almost every point $x$.
\item[(b)] The function $f^\ast$ is $\mu$-integrable and $T$-invariant.
\item[(c)] $\ds\int f^{*}\, d\mu \leq \underset{n\rightarrow \infty}{\liminf}\,\frac{1}{n}\ds\sum_{j=0}^{n-1} \ds\int f \circ T^j \, d\mu \leq \ds\int f\, d\mu.$
\item[(d)] If $\mu(X)<+\infty$, then $\ds\int f^{*}\, d\mu = \underset{n\rightarrow \infty}{\lim}\,\frac{1}{n}\ds\sum_{j=0}^{n-1} \ds\int f \circ T^j \, d\mu = \ds\int f\, d\mu.$
\item[(e)] If $\mu(X)<+\infty$, then $\mu$ is $T$-invariant.
\end{itemize}
\end{theorem}

\bigskip

Three comments before starting the proof:

\begin{itemize}
\item[1.] When $\mu$ is finite, the statement $(a)$ is a particular case of Chacon-Ornstein theorem \cite{CO}\footnote{
\emph{If $(X,\mathfrak{B}, \mu)$ is $\sigma$-finite, $0\leq f,g \in \mathcal{L}^1(X,\mu)$ and $\mathcal{U}:\mathcal{L}^1(X,\mu) \rightarrow \mathcal{L}^1(X,\mu)$ is a positive contraction, then the sequence $\left(\frac{\ds\sum_{j=0}^{n-1} \mathcal{U}^j(f)}{\ds\sum
_{j=0}^{n-1} \mathcal{U}^j(g)}\right)_{\mathbb{N}}$ converges to a finite limit $\mu$ almost everywhere on the set $\mathcal{C}=\{x \in X:\ds\sum_{j=0}^{n-1} \mathcal{U}^j(g)>0\}$}.}. But, in general, we are not allowed to choose $g \equiv 1$ on this statement.
\item[2.] If, besides being half-invariant, $\mu$ is ergodic, then $f^{*}$ is constant, given by the limit $\underset{n\rightarrow \infty}{\lim}\,\frac{1}{n}\ds\sum_{j=0}^{n-1} \ds\int f \circ T^j\, d\mu$, which is less or equal to $\ds\int f\, d\mu$. Ergodicity is not interesting when $\mu(X)=+\infty$, since then this limit is always zero because $f^{*}$ is $\mu$ integrable.
\item[3.] The loss in (c) is expected and not only due to the lack of invariance, but also a consequence of the non-finiteness of the measure. For instance, consider $S(x)=x+1$, $T(x)=2x$, for $x \in \mathbb{R}$, $\mu$ the Lebesgue measure (which is $S$-invariant and $T$-half-invariant) and $f=\chi_{|_{[1,2[}}$. Then, as the orbit by $T$ or $S$ of each $x \in \mathbb{R}$ crosses $[1,2[$ at most once, $f^* \equiv 0$ for both dynamics. As $T^{-j}([1,2[)=[\frac{1}{2^j},\frac{1}{2^{j-1}}[$, for all $j\geq1$, the sequence $\left(\ds\int \,f \circ T^j\, d\mu\right)_{n \in \mathbb{N}}$ has limit zero, and so does $\left(\frac{1}{n}\ds\sum_{j=0}^{n-1} \ds\int f \circ T^j\, d\mu\right)_{n \in \mathbb{N}}$. Hence
$$\ds\int f^{*}\, d\mu = 0 = \underset{n\rightarrow \infty}{\lim}\,\frac{1}{n}\ds\sum_{j=0}^{n-1} \ds\int f \circ T^j\, d\mu < \ds\int f\, d\mu=1.$$
Concerning $S$, we have $\ds\int \,f \circ S^j\, d\mu=\ds\int \,f \, d\mu=1$, for all $j\geq1$, and so
$$\ds\int f^{*}\, d\mu=0<\underset{n\rightarrow \infty}{\lim}\,\frac{1}{n}\ds\sum_{j=0}^{n-1} \ds\int f \circ S^j\, d\mu = \ds\int\, f\, d\mu=1.$$
\end{itemize}
\bigskip

To prove that the half-invariance is enough to ensure the convergence $\mu$ almost everywhere of the Birkhoff averages of $f$, we essentially had two different approaches available and looked for the one which used in less instances the invariance of the measure. To our knowledge, the simplest proof of Birkhoff's theorem is due to T. Kamae \cite{K}, where invariance only intervenes once. Yet, the argument is only valid for finite measures. The classical reasoning due to Riesz uses the invariance of the measure in several steps, but, through a Maximal Ergodic theorem, avoids the constraint of the finiteness. Our argument mixes both strategies, employing a version of the Maximal Ergodic theorem for half-invariant measures.
\end{section}

\begin{section}{Proof}
The first part of this section differs from \cite{KW2}, which is the reformulation of Kamae's proof in standard language, in a few details only; we will emphasize the differences.

First notice that it is enough to verify the pointwise convergence for a non-negative test function $f$ in $\mathcal{L}^\infty$. Otherwise, we take $f^+=\max \{f,0\}$ and $f^-=\max \{-f,0\}$ and, if $f$ belongs to $\mathcal{L}^1(X,\mu)$, we approach $f$ by the bounded maps $f_M=\min\,\{f,M\}$, where $M>0$, apply the argument to each $f_M$, let $M$ go to $+\infty$ and then use the Monotone Convergence theorem.

Consider a $\sigma$-finite measure $\mu$ half-invariant by $T$, a non-negative $f$ in $\mathcal{L}^\infty$ and the maps
$$\overline{f}(x)=\underset{n\rightarrow \infty}{\limsup}\,\frac{1}{n}\ds\sum_{j=0}^{n-1} f \circ T^j$$
$$\underline{f}(x)=\underset{n\rightarrow \infty}{\liminf}\,\frac{1}{n}\ds\sum_{j=0}^{n-1} f \circ T^j.$$
These are measurable $T$-invariant functions since
$$\frac{1}{n}\ds\sum_{j=0}^{n-1} f \circ T^{j+1}=\frac{n+1}{n}.\frac{1}{n+1}\ds\sum_{j=0}^{n} f \circ T^j-\frac{1}{n}f$$
and, by Fatou Lemma and half-invariance, $\mu$-integrable as well, with
$$0\leq \ds\int \, \underline{f} \, \leq \, \ds\int \, \overline{f} \, \leq \, \ds\int \, f.$$
We will prove that $\mu$-almost everywhere $\overline{f}=\underline{f}$. The common value of these two functions defines a map $f^{*}$ which is $T$-invariant and belongs to $\mathcal{L}^1(\mu)$: as $\mu$ is half-invariant, we have
\begin{eqnarray*}
\ds\int \, \left|\frac{1}{n}\ds\sum_{j=0}^{n-1} f \circ T^j\right| \, d\mu & \leq & \frac{1}{n}\ds\sum_{j=0}^{n-1} \ds\int \, |f|(T^j) \, d\mu \\
& \leq &  \frac{1}{n}\ds\sum_{j=0}^{n-1} \ds\int \, |f| \, d\mu \\
&=& \ds\int \, |f| \, d\mu < +\infty
\end{eqnarray*}
and so we may use the Dominated Convergence theorem. We are now due to prove assertions $(a)$, $(c)$, $(d)$ and $(e)$.

\medskip

\begin{subsection}{First case: $\mu(X)<+\infty$}
\begin{proof}

\noindent (a) Fix a non-negative $f \in \mathcal{L}^\infty$, a real $\varepsilon>0$ and $M=\sup \,\{f(x): x \in X\}$. If $M=0$, then $f^{*}=f\equiv 0$ and the proof ends. Otherwise, for each $x \in X$, take
$$n(x)=\min\,\{k \in \mathbb{N}: \overline{f}(x)\leq \frac{1}{k}\ds\sum_{j=0}^{k-1} f(T^j(x)) + \varepsilon\}.$$
Thus, as $\overline{f}\circ T=\overline{f}$,
$$\ds\sum_{j=0}^{n(x)-1} \overline{f}(T^j(x)) \leq \ds\sum_{j=0}^{n(x)-1} f(T^j(x))+n(x)\varepsilon.$$
The main problem concerning this estimate is the set of points for which the convergence is too slow. It is the union, for big $k$, of the \emph{tail-sets}
$$A_{\varepsilon,k}=\{ x \in X: n(x)>k \}.$$
However, these are measurable and

\begin{lemma}
We may find a positive integer $N_\varepsilon$ such that $\mu(A_{\varepsilon,N_\varepsilon})<\frac{\varepsilon}{M}.$
\end{lemma}

\begin{proof}
If there were an $\varepsilon>0$ such that, for every $N\in \mathbb{N}$, we had $\mu(A_{\varepsilon,N_\varepsilon})\geq\frac{\varepsilon}{M}$, then, as
$$\{ x \in X: n(x)> N+1 \} \subseteq \{ x \in X: n(x)> N \},$$
by the Monotone Convergence theorem we would deduce (since $\mu(X)<\infty$) that the measurable set
$$B=\bigcap_{N=1}^\infty \,\, \{ x \in X: n(x)> N \}$$
would verify
$$\mu(B)=\underset{n\rightarrow \infty}{\lim} \,\mu (\{ x \in X: n(x)> N \})\geq \frac{\varepsilon}{M},$$
and so $B\neq \emptyset$ and any $b \in B$ would have $n(b)=+\infty$.
\end{proof}

\bigskip

This Lemma enables us to replace $f$ and $n$ by a map that has better tail-sets and bounded ${n}$.

\begin{definition}
$$\tilde{f}_{N_\varepsilon}(x)=\left\{ \begin{array}{cc}
f(x) & \text{ if } x \notin A_{\varepsilon,N_\varepsilon} \\
M & \text{ otherwise }
\end{array}
\right.
$$

\medskip

$$\tilde{n}(x)=\left\{ \begin{array}{cc}
n(x) & \text{ if } x \notin A_{\varepsilon,N_\varepsilon} \\
1 & \text{ otherwise }
\end{array}
\right.
$$
\end{definition}

\noindent Observe that $\tilde{n}(x)\leq N_\varepsilon$, for any $x$. And it is easy to deduce that

\begin{lemma}\cite{KW2}
For every $x$,
$$\ds\sum_{j=0}^{\tilde{n}(x)-1} \overline{f}(T^j(x)) \leq \ds\sum_{j=0}^{\tilde{n}(x)-1} \tilde{f}_{N_\varepsilon}(T^j(x))+\tilde{n}(x)\varepsilon.$$
\end{lemma}

\noindent Take then a positive integer $L_\varepsilon$ such that $\frac{N_\varepsilon \, M}{L_\varepsilon}<\varepsilon$ and apply the previous Lemma to the first summands while upper-bounding the others by $M$. We then get, for all $L\geq L_\varepsilon$,

\begin{lemma}\cite{KW2}
$$\ds\sum_{j=0}^{L-1} \overline{f}(T^j(x)) \leq \ds\sum_{j=0}^{L-1} \tilde{f}_{N_\varepsilon}(T^j(x))+L\varepsilon + (N_\varepsilon -1)M.$$
\end{lemma}

\noindent Hence, dividing by $L$, this estimate yields
$$\frac{1}{L} \ds\sum_{j=0}^{L-1} \overline{f}(T^j(x)) \leq \frac{1}{L}\ds\sum_{j=0}^{L-1} \tilde{f}_{N_\varepsilon}(T^j(x))+2\varepsilon$$
and therefore, for any $x \in X$ and $L\geq L_\varepsilon$, we have
\begin{equation}
\overline{f}(x) \leq \frac{1}{L}\ds\sum_{j=0}^{L-1} \tilde{f}_{N_\varepsilon}(T^j(x))+2\varepsilon. \label{supremo}
\end{equation}
\medskip

We now repeat this argument with $\underline{f}$. For each $x \in X$, take
$$m(x)=\min\,\{k \in \mathbb{N}: \frac{1}{k}\ds\sum_{j=0}^{k-1} f(T^j(x)) \leq \underline{f}(x) + \varepsilon\}$$
select the tail-sets
$$C_{\varepsilon,k}=\{ x \in X: m(x)>k \}$$
and find a positive integer $J_\varepsilon$ such that $\mu(C_{\varepsilon,J_\varepsilon})<\frac{\varepsilon}{M}.$ Then replace $f$ and $m$ by
\begin{definition}

$$\check{f}_{J_\varepsilon}(x)=\left\{ \begin{array}{cc}
f(x) & \text{ if } x \notin C_{\varepsilon,J_\varepsilon} \\
0 & \text{ otherwise }
\end{array}
\right.
$$

\medskip

$$\tilde{m}(x)=\left\{ \begin{array}{cc}
m(x) & \text{ if } x \notin C_{\varepsilon,J_\varepsilon} \\
1 & \text{ otherwise }
\end{array}
\right.
$$

\end{definition}

\noindent As before, we may fix a positive integer $L^\prime_\varepsilon$ such that, for all $L\geq L^\prime_\varepsilon$ and all $x \in X$, we have
\begin{equation}
\underline{f}(x) \geq \frac{1}{L}\ds\sum_{j=0}^{L-1} \check{f}_{J_\varepsilon}(T^j(x))-2\varepsilon. \label{infimo}
\end{equation}

\noindent Nevertheless, as $\tilde{f}$ does not, in general, coincide with $\check{f}$, we cannot use the estimates (\ref{supremo}) and (\ref{infimo}) to conclude that $\underline{f}\equiv \overline{f}$. But we may integrate these two inequalities, for $L\geq \max \, \{L_\varepsilon, L^\prime_\varepsilon\}$, taking into account that

\begin{lemma}\label{integrais}
For any $j \in \mathbb{N}_0$,
$$\ds\int\, \tilde{f}_{N_\varepsilon}\circ T^j \, d\mu \leq \ds\int\, f \circ T^j \, d\mu + \varepsilon.$$
\end{lemma}

\begin{proof} Given such a $j$, as $\mu(T^{-j}(A_{\varepsilon,N_\varepsilon})) \leq \mu(A_{\varepsilon,N_\varepsilon}) <\frac{\varepsilon}{M}$,
\begin{eqnarray*}
\ds\int\, \tilde{f}_{N_\varepsilon}\circ T^j \, d\mu &=& \ds\int_{X\setminus T^{-j}(A_{\varepsilon,N_\varepsilon})}\, \tilde{f}_{N_\varepsilon}\circ T^j \, d\mu + \ds\int_{T^{-j}(A_{\varepsilon,N_\varepsilon})}\, \tilde{f}_{N_\varepsilon}\circ T^j \, d\mu \\
 &\leq& \ds\int_{X\setminus T^{-j}(A_{\varepsilon,N_\varepsilon})}\, f\circ T^j \, d\mu + M \, \mu\left(T^{-j}(A_{\varepsilon,N_\varepsilon})\right) \\
 &\leq& \ds\int\, f \circ T^j \, d\mu \,+ \varepsilon.
\end{eqnarray*}
\end{proof}

\bigskip

Thus, from (\ref{supremo}), we deduce that
$$\ds\int\, \overline{f}\, d\mu \leq \frac{1}{L}\ds\sum_{j=0}^{L-1} \ds\int \, \tilde{f}_{N_\varepsilon}\circ T^j+2\,\varepsilon \,\mu(X)$$
and, with Lemma \ref{integrais}, that
$$\ds\int\, \overline{f}\, d\mu \leq \frac{1}{L}\ds\sum_{j=0}^{L-1} \ds\int f\circ T^j+2\,\varepsilon \,\mu(X) + \varepsilon.$$
So, as $\varepsilon$ is arbitrary, we get
\begin{equation}
\ds\int\, \overline{f}\, d\mu \leq \liminf_{n \rightarrow +\infty}\,\frac{1}{n}\ds\sum_{j=0}^{n-1} \ds\int \, f\circ T^j \, d\mu.  \label{overline}
\end{equation}
Similarly, given $j$, we have
\begin{eqnarray*}
\ds\int\, \check{f}_{J_\varepsilon}\circ T^j \, d\mu &=& \ds\int_{X\setminus T^{-j}(C_{\varepsilon,J_\varepsilon})}\, \check{f}_{J_\varepsilon}\circ T^j \, d\mu + \ds\int_{T^{-j}(C_{\varepsilon,J_\varepsilon})}\, \check{f}_{J_\varepsilon}\circ T^j \, d\mu \\
&=& \ds\int_{X\setminus T^{-j}(C_{\varepsilon,J_\varepsilon})}\, f\circ T^j \, d\mu \\
&\geq& \ds\int \, f\circ T^j \, d\mu - M \, \mu\left(T^{-j}(C_{\varepsilon,J_\varepsilon})\right)\\
&\geq& \ds\int \, f\circ T^j \, d\mu \,- \varepsilon
\end{eqnarray*}
so, from (\ref{infimo}), we conclude that
\begin{eqnarray*}
\ds\int\, \underline{f}\, d\mu &\geq& \frac{1}{L}\ds\sum_{j=0}^{L-1} \ds\int \, \check{f}_{J_\varepsilon}\circ T^j-2\,\varepsilon \,\mu(X) \\
&\geq& \frac{1}{L}\ds\sum_{j=0}^{L-1} \ds\int \, f\circ T^j - 2\,\varepsilon \,\mu(X) - \varepsilon.
\end{eqnarray*}
Again, as $\varepsilon$ is arbitrary,
\begin{equation}
\ds\int\, \underline{f}\, d\mu \geq \limsup_{n \rightarrow +\infty}\,\frac{1}{n}\ds\sum_{j=0}^{n-1} \ds\int \, f\circ T^j \, d\mu.  \label{underline}
\end{equation}
Finally, the inequalities (\ref{overline}) and (\ref{underline}) yield
$$\ds\int\, \underline{f} \geq \ds\int\, \overline{f}$$
which, as $\underline{f}\leq\overline{f}$, implies that
$$\underline{f}(x)=\overline{f}(x) \,\,\, \mu \text{ almost every } x.$$

\bigskip

\noindent (d) Moreover, (\ref{overline}) and (\ref{underline}) also ensure that
$$\ds\int\, f^* \, d\mu = \lim_{n \rightarrow +\infty}~\,\frac{1}{n}\ds\sum_{j=0}^{n-1} \ds\int \, f\circ T^j \, d\mu$$
and so, as $\mu$ is half-invariant,
$$\ds\int\, f^* \, d\mu \leq \ds\int \, f \, d\mu.$$

\bigskip

\noindent (e) If $\mu$ is a finite half-invariant measure, then, given a measurable $B$,
\begin{eqnarray*}
\mu(T^{-1}(B) & = & \mu\left(T^{-1}\left(X\setminus(X\setminus B)\right)\right) \\
& = & \mu\left(T^{-1}(X)\setminus T^{-1}(X\setminus B)\right)\\
& = & \mu\left(X\setminus T^{-1}(X\setminus B)\right)\\
& = & \mu(X)- \mu\left(T^{-1}(X\setminus B)\right)\\
& \geq & \mu(X) - \mu\left(X\setminus B\right)\\
& = & \mu\left(B)\right)
\end{eqnarray*}
thus $\mu$ is $T$-invariant.
\end{proof}

\bigskip

\begin{corollary}
Let $(X,\mathfrak{B}, \mu)$ be a finite measure space, $T:X \rightarrow X$ a measurable transformation and $(\lambda_j)_{j \in \mathbb{N}}$ a sequence of non-negative real numbers whose arithmetical averages are upperbounded by some $\lambda > 0$. If, for each $B \in \mathfrak{B}$ and each $j\in \mathbb{N}$, we have
$$\mu(T^{-j}(B)\leq \lambda_j\mu(B),$$
then, for any non-negative $f \in \mathcal{L}^\infty(X,\mu)$,
\begin{itemize}
\item[(i)] $f^\ast(x) = \underset{n\rightarrow{\infty}}{\text{lim}} \, \frac{1}{n} \ds\sum_{j=0}^{n-1} f (T^j(x))$ exists for $\mu$ almost every point $x$.
\item[(2i)] $\ds\int f^{*}\, d\mu = \underset{n\rightarrow \infty}{\lim}\,\frac{1}{n}\ds\sum_{j=0}^{n-1} \ds\int f\circ T^j\, d\mu \leq \lambda \, \ds\int f\, d\mu.$
\end{itemize}
\end{corollary}
\end{subsection}

\bigskip

\begin{subsection}{Second case: $\mu(X)=+\infty$}
\begin{proof}

\noindent (a) Let $(X_n)_{n \in \mathbb{N}}$ be an increasing sequence of measurable sets such that $\mu(X_n) < +\infty$ for all $n$ and $X=\bigcup^{\infty}_{n=1} \, X_n$. Consider $\alpha < \beta$ and the set
$$Y_{\alpha,\beta}=\{x \in X: \underline{f}(x)<\alpha<\beta< \overline{f}(x)\}.$$
If, for all $\alpha$ and $\beta$, we have $\mu(Y_{\alpha,\beta})=0$, then we may conclude the pointwise convergence from the equality
$$\mu(\{x \in X: \underline{f}(x)\neq \overline{f}(x)\})=\mu(\bigcup_{\alpha, \beta \,\in\, \mathbb{Q}:\,\, \alpha<\beta}\,Y_{\alpha,\beta})=0.$$
The set $Y_{\alpha,\beta}$ is measurable and $T^{-1}(Y_{\alpha,\beta})=Y_{\alpha,\beta}$. To verify that $\mu(Y_{\alpha,\beta})=0$, we will check how big are the subsets of $Y_{\alpha,\beta}$ with finite measure (which exist since $\mu$ is $\sigma$-finite).

\begin{proposition}\label{subsets}
Fix a pair $\alpha < \beta$ and assume that $\beta>0$.\footnote{Otherwise, $\alpha <0$ and we may take $-f$, $-\alpha$ and $-\beta$ instead.} Then:
\begin{itemize}
\item[(i)] $C\subseteq Y_{\alpha,\beta} \text{ and } \mu(C)<+\infty \Rightarrow \mu(C)\leq \frac{1}{\beta} \ds\int\, |f| \, d\mu$.
\item[(2i)] $\mu(Y_{\alpha,\beta})=0.$
\end{itemize}
\end{proposition}

\medskip

\begin{proof}

\noindent (i) This is a consequence of Hopf's Maximal Ergodic theorem, adapted to half-invariant measures.

\begin{lemma}{\cite{H, Ga}}\label{MET}
Let $\mu$ be a $\sigma$-finite measure and $\mathcal{U}:\mathcal{L}^1(X,\mu) \rightarrow \mathcal{L}^1(X,\mu)$ a linear operator which is positive ($g\geq 0 \Rightarrow \mathcal{U}(g) \geq 0$) and contractive ($\forall \, g \in \mathcal{L}^1(X,\mu) \, \, \|\mathcal{U}(g)\|_1 \leq \|g\|_1$). Then
$$ \forall \, f \in \mathcal{L}^1(X,\mu) \, \,\,\, \ds\int_{\{\hat{f}\,>\,0\}} \, f d\mu \geq 0$$
where $\hat{f}= \sup_{n \geq 1} \, \frac{1}{n} \ds\sum^{n-1}_{j=0} \mathcal{U}^j(f)$ and $\{\hat{f}>0\}=\{x \in X : \hat{f}(x)>0\}$.
\end{lemma}

\medskip

\noindent Consider one such a subset $C$, the map $f-\beta \, \chi_C$, which is in $\mathcal{L}^1(X,\mu)$ since $\mu(C)<+\infty$, and the operator $\mathcal{U}:g \mapsto g\circ T$. Applying Lemma \ref{MET}, we conclude that
$$\ds\int_{\left\{\widehat{f-\beta \, \chi_C}\,\,>0\right\}} \, \, (f-\beta \, \chi_C) \,\, d\mu \geq 0.$$
Moreover,
\begin{lemma}
$Y_{\alpha,\beta}\subseteq \left\{\widehat{f-\beta \, \chi_C}>0\right\}.$
\end{lemma}

\begin{proof}
Take $x \in Y_{\alpha,\beta}$. As $\beta < \overline{f}(x)$, at least one (in fact infinitely many) averages $\frac{1}{n}\ds\sum_{j=0}^{n-1} f(T^j(x))$ are strictly bigger than $\beta$. So, for at least one $n$, we have
$$\ds\sum_{j=0}^{n-1} \left(f(T^j(x))-\beta\,\chi_C(T^j(x))\right)\geq \ds\sum_{j=0}^{n-1} \left(f(T^j(x))-\beta\right)=\left(\ds\sum_{j=0}^{n-1} f(T^j(x))\right)-n\beta \geq 0.$$
\end{proof}

\bigskip

Therefore
\begin{eqnarray*}
0 & \leq & \ds\int_{\left\{\widehat{f-\beta \, \chi_C}\,\,>0\right\}} \, \, (f-\beta \, \chi_C) \, d\mu \\
&=& \ds\int_{\left\{\widehat{f-\beta \, \chi_C}\,\,>0\right\}} \, \, f \, d\mu -\beta \mu(C) \\
& \leq & \ds\int \, |f| \, d\mu -\beta \mu(C).
\end{eqnarray*}

\bigskip

\noindent (2i) Firstly, from (i) and the Monotone Convergence theorem, we obtain
$$\mu(Y_{\alpha,\beta})= \lim_{n \rightarrow +\infty} \, \mu(Y_{\alpha,\beta}\cap X_n) \leq \ds\int\, |f| \, d\mu < +\infty.$$
Then, as $Y_{\alpha,\beta}$ is $T$-invariant, if $\mu(Y_{\alpha,\beta})$ were positive, we might restrict the dynamics to $Y_{\alpha,\beta}$ and apply the first part of this proof to $T_{|_{Y_{\alpha,\beta}}}$, the measure $\nu=\frac{\mu}{\mu(Y_{\alpha,\beta})}$, which is half-invariant as well, and any $g \in \mathcal{L}^{\infty}(Y_{\alpha,\beta},\nu)$. But for $g=f_{|_{Y_{\alpha,\beta}}}$, the pointwise convergence of the time averages fails on all points of $Y_{\alpha,\beta}$. Hence $\mu(Y_{\alpha,\beta})=0$.
\end{proof}

\bigskip

\noindent (c) As the maps $F_n= \frac{1}{n} \ds\sum_{j=0}^{n-1} f (T^j)$ are non-negative and $\ds\int \, |F_n| \, d\mu \leq \ds\int \, |f| \, d\mu \leq+\infty$, we have
$\underset{n\rightarrow \infty}{\liminf} \,\ds\int \, F_n \, d\mu <+\infty$ and so, by Fatou Lemma,
$$\ds\int \, \underset{n\rightarrow \infty}{\liminf} \,F_n \, d\mu \leq \underset{n\rightarrow \infty}{\liminf} \,\ds\int \, F_n \, d\mu $$
that is,
$$\ds\int \, f^* \, d\mu \leq \underset{n\rightarrow \infty}{\liminf}\, \ds\int \, \frac{1}{n} \ds\sum_{j=0}^{n-1} f \circ T^j d\mu.$$
Since $\mu$ is half-invariant and $f\geq 0$, we also have
$$ \forall \,\, j \in \mathbb{N}\,\,\,\,\,\,\,\ds\int \, f\circ T^j \, d\mu \leq \ds\int \, f \, d\mu$$
and so
$$\ds\int \, f^* \, d\mu \leq \ds\int \, f \, d\mu.$$
\end{proof}
\end{subsection}
\end{section}

\bigskip

\flushleft
\emph{Maria Carvalho} \ \  (mpcarval@fc.up.pt)\\
\emph{Fernando Moreira} \ \  (fsmoreir@fc.up.pt)\\
CMUP and Departamento de Matem\'atica \\
Rua do Campo Alegre, 687 \\ 4169-007 Porto \\ Portugal\\


\medskip

\begin{thebibliography}{00}
\bibitem{C} R. V. Chacon, \emph{A class of linear transformations} Proc. Amer. Math. Soc. 15 (1964) 560--564
\bibitem{CO} R. V. Chacon, D. S. Ornstein, \emph{A general ergodic theorem} Illinois J. Math. 4 (1960) 153--60
\bibitem{G} A. Garsia, \emph{A simple proof of E. Hopf maximal ergodic theorem} J. Math. Mech. 14 (1965) 381--2
\bibitem{Ga} A. Gaunersdorfer, \emph{Time averages for heteroclinic attractors} SIAM J. Appl. Math. 52 (1992) 1476--1489
\bibitem{H} P.R Halmos, \emph{An ergodic theorem} Proc. Nat. Acad. Sci. U.S.A. vol. 32 (1946) 156-161
\bibitem{Ha} P.R Halmos, \emph{Lectures on Ergodic Theory} Chelsea Publishing Company (1956)
\bibitem{Hg} G. Helmberg, \emph{On the converse of Hopf's ergodic theorem} Z. Wahrscheinlichkeitstheorie verw. Gebiete 21 (1972) 77--80
\bibitem{Ho} E. Hopf, \emph{The general temporally discrete Markoff process} J. Rat. Mech. Anal. 3 (1954) 13--45
\bibitem{Hu} W. Hurewicz, \emph{Ergodic theorem without invariant measure} Annals of Math. vol. 45, Nº 1 (1944) 192--206
\bibitem{K} T. Kamae, \emph{A simple proof of the ergodic theorem using nonstandard analysis} Isr. J. Math. 42 (1982) 284--290
\bibitem{Kh} A. Khintchine \emph{Fourierkoeffizienten l$\ddot{a}$ngs einer Bahn im Phasenraum} Rec. Math. (Mat. Sbornik) vol. 41 (1934) 14--15
\bibitem{KW} Y. Katznelson, B. Weiss, \emph{The construction of quasi-invariant measures} Isr. J. Math. 12 (1972) 1--4
\bibitem{KW2} Y. Katznelson, B. Weiss, \emph{A simple proof of some ergodic theorems} Isr. J. Math. 42 (1982) 291--296
\bibitem{Kr} W. Krieger, \emph{On quasi-invariant measures in uniquely ergodic systems} Invent. Math. 14 (1971) 184--196
\bibitem{LS} M. Lin, R. Sine, \emph{The individual ergodic theorem for non-invariant measures} Z. Wahrscheinlichkeitstheorie verw. Gebiete 38 (1977) 329--331
\bibitem{M} R. Ma\~n\'e, \emph{Ergodic theory and differentiable dynamics} Springer-Verlag (1987)
\bibitem{Ru} R. Rudnicki, \emph{Markov operators: applications to diffusion processes and population dynamics} Applicationes Mathematicae 27, 1 (2000) 67--69
\bibitem{T} F. Takens, \emph{Heteroclinic attractors: time averages and moduli of topological conjugacy} Bol. Soc. Brasil. Mat. (N.S.) 25, Nº 1 (1994) 107-–120
\bibitem{WW} N. Wiener, A. Wintner \emph{Harmonic analysis and ergodic theory} Amer. J. Math., vol. 63 (1941) 415--426

\end{thebibliography}
\end{document}